\documentclass[a4paper, reqno]{amsart}

\usepackage[style=alphabetic, 
maxbibnames=99, firstinits]{biblatex} 
\renewbibmacro{in:}{%
  \ifboolexpr{%
     test {\ifentrytype{article}}%
     or
     test {\ifentrytype{inproceedings}}%
  }{}{\printtext{\bibstring{in}\intitlepunct}}%
}
\usepackage[margin=1.1in]{geometry} 

\usepackage{lmodern}
\usepackage{microtype}

\usepackage{amsmath}        
\usepackage{amsfonts}       
\usepackage{amsthm}         
\usepackage{amssymb}%
\usepackage{bm}             
\usepackage[dvipsnames]{xcolor}  
\usepackage[colorlinks=true]{hyperref}
    \hypersetup{citecolor=Sepia, linkcolor=Sepia, urlcolor=MidnightBlue}
\usepackage{cleveref}
\usepackage{mathtools}
\usepackage[shortlabels]{enumitem}
	\setlist[enumerate,1]{label=(\alph*), font=\normalfont}
    \setlist[enumerate,2]{label=(\roman*), font=\normalfont,, ref=\theenumi(\roman*)}
    
\usepackage{csquotes}
\usepackage{fontawesome}
\usepackage{mdframed}

\numberwithin{equation}{section}

\theoremstyle{plain}
\newtheorem{theorem}{Theorem}[section]
\newtheorem{lemma}[theorem]{Lemma}
\newtheorem{proposition}[theorem]{Proposition}
\newtheorem{corollary}[theorem]{Corollary}

\theoremstyle{definition}
\newtheorem{remark}[theorem]{Remark}
\newtheorem*{convention}{Convention}

\newcommand{\R}{\mathbb{R}}

\newcommand{\Z}{\mathbb{Z}}
\newcommand{\Q}{\mathbb{Q}}

\newcommand{\vct}[1]{\bm{#1}} 
\newcommand{\e}{\vct{e}}
\newcommand{\f}{\vct{f}}
\newcommand{\vu}{\vct{u}}
\newcommand{\vv}{\vct{v}}

\newcommand{\ip}{\mathfrak{p}}
\newcommand{\ia}{\mathfrak{a}}

\newcommand{\cP}{\mathcal{P}}

\newcommand{\ve}{\varepsilon}

\newcommand{\abs}[1]{\left|#1\right|}

\newcommand{\co}{\mathcal{O}}
\newcommand{\U}{\mathcal{U}}
\newcommand{\UPlus}{\U^+}
\newcommand{\OK}[1][K]{\co_{{\:\!\! #1}}}
\newcommand{\OKPlus}[1][K]{\OK[#1]^+}
\newcommand{\UK}[1][K]{\mathcal{U}_{#1}} 
\newcommand{\UKPlus}[1][K]{\mathcal{U}_{#1}^+} 
\newcommand{\UKctv}[1][K]{\mathcal{U}_{#1}^{2}} 

\newcommand{\qf}[1]{\langle#1\rangle} 

\newcommand{\Tr}[2]{\mathrm{Tr}_{#1}\left(#2\right)} 
\newcommand{\norm}[2]{\mathrm{N}_{#1}\left(#2\right)} 

\newcommand{\sg}{\bm{\mathrm{sg}}}
\newcommand{\sign}{\mathrm{sgn}\,}

\def\house#1{{%
    \setbox0=\hbox{$#1$}
    \vrule height \dimexpr\ht0+1.4pt width .4pt depth \dp0\relax
    \vrule height \dimexpr\ht0+1.4pt width \dimexpr\wd0+2pt depth \dimexpr-\ht0-1pt\relax
    \llap{$#1$\kern1pt}
    \vrule height \dimexpr\ht0+1.4pt width .4pt depth \dp0\relax
}}

\newcommand*\quot[2]{{^{\textstyle #1}\big/_{\textstyle #2}}} 

\title{Kitaoka's Conjecture and sums of squares}
\author{V\'it\v{e}zslav Kala}
\address{Charles University, Faculty of Mathematics and Physics, Department of Algebra, Sokolov\-sk\' a~83, 18600 Praha~8, Czech Republic}
\email{vitezslav.kala@matfyz.cuni.cz}

\author{Krist\'yna Kramer}
\address{Charles University, Faculty of Mathematics and Physics, Department of Algebra, Sokolov\-sk\' a~83, 18600 Praha~8, Czech Republic}
\email{kristyna.kramer@gmail.com}

\author{Jakub Kr\'asensk\'y}
\address{Czech Technical University in Prague, Faculty of Information Technology, Department of Applied Mathematics, Thákurova~9, 16000 Praha~6, Czech Republic}
\email{jakub.krasensky@fit.cvut.cz}
\thanks{V.K. was supported by Czech Science Foundation grant 26-20514S. K.K. was supported by Charles University project PRIMUS/25/SCI/008.}

\keywords{universal quadratic form, quadratic lattice, number field, totally real, Kitaoka's Conjecture, sum of squares}
\subjclass[2020]{11E12, 11E20, 11E25, 11R04, 11R16, 11R80}

\begin{document}

\begin{abstract}
We connect the existence of a ternary classical universal quadratic form over a totally real number field $K$ with the property that all totally positive multiples of 2 are sums of squares (if $K$ does not contain $\sqrt 2$ or contains a nonsquare totally positive unit). In particular, we get that Kitaoka's Conjecture holds for all fields of odd discriminant.
\end{abstract}

\maketitle

\section{Introduction}

In 1770, Lagrange proved the \emph{Four Square Theorem} that says that every positive integer is represented by the quadratic form $x^2+y^2+z^2+w^2$. This was followed by wide interest in universal quadratic forms, at first over the rationals and their integers $\Z$, and then over number fields. 

To be more precise, let $K$ be a totally real number field with ring of integers $\OK$. An element $\alpha\in K$ is \emph{totally positive}, $\alpha\succ 0$, if $\sigma(\alpha)>0$ for all embeddings $\sigma:K\hookrightarrow \R$; and $\OKPlus$ denotes the set of totally positive elements of $\OK$. By a \emph{positive definite $r$-ary quadratic form $Q$} we mean $Q(x_1,\dots,x_r)=\sum_{1\leq i\leq j\leq r} \alpha_{ij}x_ix_j$ with $\alpha_{ij}\in\OK$ such that  $Q(\vv)\succ 0$ for all nonzero $\vv\in K^r$. Finally, we say that  $Q$ is \emph{universal over $K$} if it is positive definite  and represents all totally positive algebraic integers in $K$, i.e., if for all $\alpha\in\OKPlus$, there is some $\vv\in\OK^{\;\!r}$ such that $Q(\vv)=\alpha$.

Thus the Four Square Theorem amounts to saying that $x^2+y^2+z^2+w^2$ is a universal quadratic form over $\Q$. Over the rationals, many mathematicians worked on classifying universal forms over the last century, including Ramanujan, Dickson, Willerding, and then Conway--Schneeberger and Bhargava--Hanke with their celebrated \emph{15}- and \emph{290}-\emph{Theorems}. 

As one can observe already  over $\Q$, when the degree $[K:\Q]$ is odd, then there cannot be a universal ternary quadratic form for local reasons \cite[Lemma~3]{EK}. However, over fields of even degree, universal forms of rank $r=3$ can exist, as first shown by Maass \cite{Ma} who established in 1941 the universality of the sum of three squares $x^2+y^2+z^2$ over $\Q(\!\sqrt 5)$. Soon after, Siegel \cite{Si} proved that there are no other totally real number fields over which the sum of any number of squares would be universal; however, more complicated universal quadratic forms can always be constructed, for example by using the \emph{Asymptotic Local--Global Principle} of Hsia--Kitaoka--Kneser \cite{HKK}.

Nevertheless, the existence of \emph{ternary} universal forms is much more problematic, and Kitaoka in the early 1990s formulated his influential conjecture that: \emph{There are only finitely many totally real number fields $K$ that admit a universal ternary classical quadratic form}. 
While Kitaoka's Conjecture has been extended also to non-classical forms, the \emph{classical} assumption (that $2\mid \alpha_{ij}$ for all $1\leq i< j\leq r$) remains its most common setting. Specifically, in 1996, Chan--Kim--Raghavan \cite{CKR} proved that the only real quadratic fields with such a form are $\Q(\!\sqrt 2), \Q(\!\sqrt 3),$ and $\Q(\!\sqrt 5)$. In 2020, Tinkov\'a with two present authors, Kr\' asensk\' y and Kramer (prev. Zemkov\'a),  showed in \cite{KTZ} that no biquadratic field admits a universal ternary classical form. A large number of authors studied universal forms under various other restrictions, see, e.g., \cite{BK1, CO, HHX, KKP, KT, Man, XZ, Ya} or the surveys \cite{Kala-survey,Kim-survey}.

However, the only other results on Kitaoka's Conjecture as such are its weak version by Kala--Yatsyna \cite{KY-weak} that assumes that the degree $[K:\Q]$ is bounded (first established by B. M. Kim for classical forms in an unpublished manuscript), and a very recent (partly conjectural) characterization of real quadratic fields with a non-classical universal ternary form \cite{KK+}. 

We shed new light on Kitaoka's Conjecture by completely resolving it for fields of odd discriminant.

\begin{theorem}\label{Th:unramified}
Let $K$ be a totally real number field where $2$ is unramified. Then $K$ admits a universal ternary classical quadratic form if and only if $K=\Q(\!\sqrt5)$.
\end{theorem}

To discuss our further results and the proof strategy, let us observe that every totally positive unit splits off from a universal classical form (see Lemma \ref{Lemma:UnitsSplit}). Thus, if we denote by $\UK$ the unit group of $K$, by $\UKPlus$ the subgroup of totally positive units, and by $\UKctv$ the subgroup of all the squares of units, then Kitaoka's Conjecture trivially holds for fields satisfying $\abs{\UKPlus/\UKctv}\geq 4$ (see Proposition \ref{Prop:TooManyUnits}). 
This leaves us with two fundamental cases of $\abs{\UKPlus/\UKctv}=1,2$, i.e., either all totally positive units are squares, or there is exactly one class of nonsquare totally positive units.

Beside proving Theorem \ref{Th:unramified}, 
we also establish the following criterion that weakens the unramified assumption to just $\sqrt 2\not \in K$ and also treats the case $\abs{\UKPlus/\UKctv}=2$. 

\begin{theorem}\label{Th:MainKitSqs}
Let $K$ be a totally real number field that admits a universal ternary classical quadratic form and such that 
\begin{enumerate}[(A)]
    \item $\sqrt2\notin K$ or \label{Th:MainKitSqs-sqrt2}
    \item $\abs{\UKPlus/\UKctv}\geq2$. \label{Th:MainKitSqs-units} 
\end{enumerate}
Then every element of $2\OKPlus$ can be written as the sum of four squares of elements of $\OK$.

Moreover, the field $K$ cannot be quartic, i.e., $[K:\Q]\neq 4$.
\end{theorem}

The proofs of Theorems \ref{Th:unramified} and \ref{Th:MainKitSqs} are summarized in Section \ref{Sec:Proof}. In Section \ref{Sec:Cases}, we look at condition \ref{Th:MainKitSqs-units} of Theorem \ref{Th:MainKitSqs}. In fact, we show that this case cannot happen when $\sqrt 2\in K$, i.e., in a field admitting a universal ternary form, we have ``$\abs{\UKPlus/\UKctv}\geq2 \,\Rightarrow\, \sqrt2\notin K$'' (establishing this is perhaps the hardest proof in the paper). 
Thus it suffices to consider 
case \ref{Th:MainKitSqs-sqrt2}, which we do in Section \ref{Sec:KitaokaSquares}. Here, we also show Proposition \ref{prop:lambda} about a special quadratic form of rank four, namely: if $\lambda$ is an arbitrary nonsquare indecomposable (in a field $K$ admitting a universal ternary form), then $\qf{1,1,\lambda,\lambda}$ represents all elements of $\lambda\OKPlus$.  

The conclusion of Theorem \ref{Th:MainKitSqs} that $[K:\Q]\neq 4$ is proven in Section \ref{Sec:Deg4}. It follows from the results of Kala--Yatsyna \cite[Thm.~1.1]{KY-EvenBetter} who found all the fields of degree $\leq 5$ where every element of $2\OKPlus$ is the sum of squares. There are only two such fields in degree 4 and we can easily verify that they do not admit a universal ternary form. Let us emphasize that the conclusion of Theorem \ref{Th:MainKitSqs} that $2\OKPlus \subset \sum \square$ is very strong: Kala--Yatsyna did not find any such fields of degree 5, and it is well possible that there are actually no such fields in degrees $>4$.

As is common in the literature, we carry most of our arguments in the geometric language of \emph{quadratic lattices}, see Section \ref{Sec:2} for definitions and Theorem \ref{Th:MainConditions} for the corresponding main result. 
However, some of our proofs do use the assumption that $Q$ is a quadratic form (i.e., a free lattice) and it is not clear how to circumvent it to get proofs for non-free lattices. In a similar way, techniques from \cite{CKR} or \cite{KTZ} do not extend in a straightforward way to non-free lattices.

Despite providing significant further evidence towards the general validity of Kitaoka's Conjecture, our results still leave open several important directions for further research. In the follow-up paper \cite{KK}, we tackle the case of general quartic fields (for which only the case $\sqrt 2\in K$ remains open) and fully resolve it.
To prove the full Kitaoka's Conjecture, there are two distinct cases $\abs{\UKPlus/\UKctv}=1$ and $2$ to consider; resolving one or the other would be a big step forward. The results of this paper should be helpful for this -- in particular, Theorem \ref{Th:MainConditions} contains a lot of information about the structure of $K$ in the latter case. Finally, note that even the (non)existence of \emph{diagonal} universal ternary forms is still open.

\section*{Acknowledgments}

We are grateful to Dayoon Park and Pavlo Yatsyna for our interesting discussions in the early stages of this project, and to the anonymous referee for a number of very useful comments that helped to improve the presentation of the paper.

\section{Preliminaries}\label{Sec:2} 

Throughout the paper, $K$ denotes a totally real number field of degree $d$ over $\Q$, and $\OK$ the ring of algebraic integers of $K$. We use $\square$ to denote an arbitrary square in $\OK$, i.e., an element $\alpha^2$ with $\alpha\in\OK$ (equivalently, with $\alpha\in K$, as the ring $\OK$ is integrally closed in $K$). We do not fix the value of $\square$, i.e., two occurrences of the symbol $\square$ stand for two (possibly different) square elements of $\OK$. In particular, the symbol $\sum\square$ stands for the set of all elements which can be written as the sum of squares of elements of $\OK$.

Since $K$ is totally real, there exist $d$ distinct embeddings $\sigma_i : K \hookrightarrow \mathbb{R}$ for $i = 1, \ldots, d$. We denote by $\mathrm{N}_{K/\Q}$ (resp. $\mathrm{Tr}_{K/\Q}$) the \emph{norm} (resp. \emph{trace}) from $K$ to $\mathbb{Q}$, defined as the product (resp. the sum) running over all embeddings.
The \textit{absolute trace} of an element is defined as $\Tr{\mathrm{abs}} \alpha = \frac{1}{[K:\Q]} \Tr{K/\Q} \alpha$; it is independent of the choice of the number field $K \ni \alpha$.

For $\alpha, \beta \in K$, we write $\alpha \succ \beta$ if $\sigma_i(\alpha) > \sigma_i(\beta)$ for all $i = 1, \ldots, d$. We use $\alpha \succeq \beta$ to denote $\alpha \succ \beta$ or $\alpha = \beta$. An element $\alpha \in K$ is called \emph{totally positive} if $\alpha \succ 0$ and \textit{totally nonnegative} if $\alpha\succeq 0$. We denote by $\OKPlus$ the set of all totally positive algebraic integers in $K$.

The group of units (i.e., the elements of norm $\pm1$) in $\OK$ is denoted by $\UK$. By Dirichlet's unit theorem, $\UK$ is a finitely generated abelian group. We denote by $\UKPlus=\UK\cap\OKPlus$ the subgroup of \emph{totally positive units}. 

Since squares are always totally positive, we have $\UKctv \subset \UKPlus$. As suggested in the introduction, the size of the factor group $\UKPlus/\UKctv$ (which is always a power of $2$) plays a significant role in this paper.

An element $\alpha \in \OKPlus$ is called \emph{indecomposable} if it cannot be written as $\alpha = \beta + \gamma$ with $\beta, \gamma \in \OKPlus$. Equivalently, $\alpha$ is indecomposable if and only if there is no $\delta\in\OKPlus$ such that $\alpha\succ\delta$. By the following lemma, all totally positive elements of small norm (so, in particular, all totally positive units) are indecomposable.

\begin{lemma}[{\cite[Lemma~2.1b]{KY-lifting}}] \label{Lemma:smallAreIndecomposable} 
Let $K$ be a totally real number field of degree $d$, and let $\alpha \in \OKPlus$. If $\norm{K/\Q}{\alpha} < 2^d$, then $\alpha$ is indecomposable.
\end{lemma}

By a \emph{decomposition of $\alpha\in\OKPlus$} we mean a decomposition of $\alpha$ into the sum of several (i.e., one or more) elements of $\OKPlus$. The rational integers $2$ and $3$ have very few decompositions regardless of the field; this fact will be very useful later.

\begin{lemma} \label{Lemma:2and3}
Let $K$ be a totally real number field. 
\begin{enumerate}
    \item  The only decompositions of $2$ are $2$ and $1+1$. \label{Lemma:2and3-2}
    \item The only decompositions of $3$ are $3$, $1+2$, $1+1+1$, and possibly $\bigl(\frac{1+\sqrt5}{2}\bigr)^2 + \bigl(\frac{1-\sqrt5}{2}\bigr)^2$. \label{Lemma:2and3-3}
\end{enumerate}
\end{lemma}

\begin{proof}
We will use a basic result on the Siegel--Schur--Smyth trace problem (see \cite[Thm.~III]{Si-trace}): if $\Tr{\mathrm{abs}} \alpha \leq \frac32$ for some $\alpha \in \OKPlus$, then $\alpha$ is $1$ or $\frac{3\pm \sqrt5}{2} = \bigl(\frac{1\pm\sqrt5}{2}\bigr)^2$; their absolute traces are $1$ and $\frac 32$, respectively.

\ref{Lemma:2and3-2} Since $\Tr{\mathrm{abs}} 2=2$, every decomposition of $2$ as the sum of two totally positive integers must contain a summand with absolute trace at most $1$. As the only such element is $1$, the result follows.

\ref{Lemma:2and3-3}
$\Tr{\mathrm{abs}} 3 = 3$, and so in every decomposition of $3$ as the sum of two totally positive integers, at least one of the summands has $\Tr{\mathrm{abs}}{\cdot} \leq \frac32$. The only such elements are $1$ and $\frac{3\pm \sqrt5}{2}$, leading to the decompositions $3=1+2$ and $3=\frac{3+ \sqrt5}{2}+\frac{3- \sqrt5}{2}$. Of course, the latter decomposition occurs only when $\sqrt5\in K$; and as $\operatorname{Tr}_{\mathrm{abs}} \bigl(\frac{3\pm \sqrt5}{2}\bigr)=\frac 32$, these elements cannot be decomposed further. In $3=1+2$, further decompositions are possible only for $2$, all of which we know by part \ref{Lemma:2and3-2}.
\end{proof}

An \emph{$r$-ary quadratic form} $Q$ over $K$ is a homogeneous polynomial of degree $2$ in $r$ variables with coefficients in $\OK$, i.e., 
\[Q(x_1,\dots,x_r)=\sum_{1\leq i\leq j\leq r} \alpha_{ij}x_ix_j \quad \text{with } \alpha_{ij}\in\OK.\]
We say that $Q$ is \emph{positive definite} if $Q(\vv)\in\OKPlus$ for every nonzero $\vv\in\OK^{\;\!r}$. Furthermore, $Q$ is \emph{classical} if $\alpha_{ij}\in 2\OK$ whenever $i\neq j$. This paper handles only classical quadratic forms, see the Convention below.

A \emph{lattice} over $K$ is a finitely generated torsion-free $\OK$-module $L$. By the structure theorem for finitely generated modules over Dedekind domains, every lattice of rank $r$ can be written as
\[
L = \OK \vv_1 \oplus \cdots \oplus \OK \vv_{r-1} \oplus \ia^{-1} \vv_r,
\]
where $\vv_1, \ldots, \vv_r\in L$ are linearly independent vectors and $\ia$ is an ideal of $\OK$. A lattice is called \emph{free} if it is free as an $\OK$-module, which occurs if and only if the class of $\ia$ in the class group is trivial. In particular, all lattices are free when $K$ has class number one.

A \emph{quadratic lattice} over $K$ is a tuple $(L,Q)$ consisting of a lattice $L$ over $K$ together with a \emph{quadratic map} $Q : L \to \OK$ satisfying:
\begin{enumerate}[(1)]
    \item $Q(\alpha \vv) = \alpha^2 Q(\vv)$ for all $\alpha \in \OK$ and $\vv \in L$;
    \item The map $B : L \times L \to K$ defined by
    \[
    B(\vu, \vv) = \tfrac{1}{2}\bigl(Q(\vu + \vv) - Q(\vu) - Q(\vv)\bigr)
    \]
    is bilinear.
\end{enumerate}
All our lattices will be quadratic, and we will usually denote them simply by $L$; unless specified otherwise, the corresponding quadratic map will always be denoted by $Q$.

We say that a lattice $L$ is \emph{positive definite} if $Q(\vv) \in \OKPlus$ for all nonzero $\vv \in L$. A lattice $L$ is called \emph{classical} if $B(\vu, \vv) \in \OK$ for all $\vu, \vv \in L$.

\medskip

\begin{mdframed}
\vspace{-2mm}
\begin{convention}
Throughout the paper, all lattices and quadratic forms are assumed to be classical and positive definite.
\end{convention}
\end{mdframed}

\medskip

Free lattices are in one-to-one correspondence with quadratic forms. For a free lattice $L$ with basis $\vv_1, \ldots, \vv_r$, the corresponding quadratic form is
\[
\varphi(x_1, \ldots, x_r) = Q(x_1 \vv_1 + \cdots + x_r \vv_r).
\]
Conversely, given a quadratic form $\varphi$, we can define a quadratic map $Q$ on the free lattice $\OK^{\;\!r}$ by $Q(x_1, \ldots, x_r) = \varphi(x_1, \ldots, x_r)$.

For $\alpha_1, \ldots, \alpha_r \in \OKPlus$, we denote by $\qf{\alpha_1, \ldots, \alpha_r}$ the diagonal quadratic form $Q(x_1, \ldots, x_r) = \alpha_1 x_1^2 + \alpha_2 x_2^2 + \cdots + \alpha_r x_r^2$ as well as the corresponding free lattice $(\OK^{\;\!r}, Q)$.

We say that a lattice $L$ \emph{represents} an element $\alpha \in \OKPlus$ if there exists a $\vv \in L$ such that $Q(\vv) = \alpha$. We write $\alpha \to L$ to denote this.

For two lattices $(L_1, Q_1)$ and $(L_2, Q_2)$, not necessarily of the same rank, 
we say that $L_1$ \emph{is represented by} $L_2$, written $L_1 \to L_2$, if there exists an $\OK$-linear map $\iota : L_1 \to L_2$ such that $Q_1(\vv) = Q_2(\iota(\vv))$ for all $\vv \in L_1$. 
Note that, for a lattice $(L,Q)$ and $\alpha\in\OKPlus$, we have $\qf{\alpha}\to L\Longleftrightarrow$ $\exists v\in L$ such that $Q(v)=\alpha\Longleftrightarrow\alpha \to L$.

Two lattices $L_1$ and $L_2$ of the same rank are called \emph{isometric}, written $L_1 \simeq L_2$, if there exists an $\OK$-linear bijection $\iota : L_1 \to L_2$ preserving the quadratic maps.

The \emph{orthogonal sum} of two lattices $(L_1, Q_1)$ and $(L_2, Q_2)$ is denoted $L_1 \perp L_2$ and is defined as the direct sum $L_1 \oplus L_2$ of the underlying modules equipped with the quadratic map $Q(\vv_1 \oplus \vv_2) = Q_1(\vv_1) + Q_2(\vv_2)$. We often exploit the fact that if $\alpha$ is indecomposable and $\alpha \to L_1 \perp L_2$, then $\alpha \to L_1$ or $\alpha \to L_2$.

A lattice $L$ over $K$ is called \emph{universal} if it represents all totally positive integers, that is, if $\alpha \to L$ for every $\alpha \in \OKPlus$.

The following lemma is useful when working with diagonal quadratic forms and, more generally, with lattices which decompose as the orthogonal sum of unary lattices.

\begin{lemma}\label{Lemma:unary}
Let $\alpha,\beta\in\OKPlus$. The following are equivalent:
 \begin{enumerate}
     \item There exists a unary lattice $L_u$ that represents both $\alpha$ and $\beta$. \label{Lemma:unary-a}
     \item $\alpha\beta=\square$ in $\OK$. \label{Lemma:unary-b}
     \item $\alpha\beta=w^2$ for some $w\in K$. \label{Lemma:unary-c}
 \end{enumerate}
\end{lemma}

\begin{proof}
Statements \ref{Lemma:unary-b} and \ref{Lemma:unary-c} are clearly equivalent thanks to $\OK$ being integrally closed.

For \enquote{$\ref{Lemma:unary-a} \Rightarrow \ref{Lemma:unary-c}$}, the vectors representing $\alpha$ and $\beta$ differ by a scalar multiple $\mu\in K$, say $\alpha=Q(\e)$ with $e\in L_u$ and $\beta=Q(\mu\e)=\mu^2Q(\e)=\mu^2\alpha$. 

For \enquote{$\ref{Lemma:unary-b} \Rightarrow \ref{Lemma:unary-a}$}, assume $\sqrt{\alpha\beta}\in\OK$. Take a vector $\e$ and define a unary quadratic map $Q$ by $Q(\e)=\frac{1}{\alpha}$. The desired lattice is $(\alpha, \sqrt{\alpha\beta})\e$; indeed, $Q(\alpha\e)=\alpha$ and $Q(\sqrt{\alpha\beta}\e)=\beta$. 
\end{proof}

As we mentioned in the introduction, we focus on fields $K$ with $\abs{\UKPlus/\UKctv}\leq2$. That is because fields with more than two classes of nonsquare totally positive units do not admit universal ternary quadratic forms for a simple reason, which we explain here. First, we need a well-known lemma that shows that units split off orthogonally from lattices -- let us stress here that the lemma requires the \emph{classical} assumption that we have in place throughout the paper.

\begin{lemma}[cf. {\cite[Prop.~3.4]{Kala-survey}}] \label{Lemma:UnitsSplit}
Let $K$ be a totally real number field and $L$ a $lattice$ over $K$. If $\ve\in\UPlus_K$ is such that $\ve\to L$, then $L\simeq\qf{\ve}\perp L'$ for some lattice $L'\subset L$.
\end{lemma}

Note that if $\ve_1,\ve_2$ are two units such that $\ve_2\to\qf{\ve_1}$, then $\ve_1$ and $\ve_2$ lie in the same class modulo $\UKctv$. Therefore, we can rule out all fields with more than two such classes. The following result is very well known, but we include its proof here for completeness.

\begin{proposition}\label{Prop:TooManyUnits}
Let $K$ be a totally real number field. Assume that there exists a universal ternary lattice over $K$. Then $\abs{\UKPlus/\UKctv}\leq2$.
\end{proposition}

\begin{proof}
If $\abs{\UKPlus / \UKctv} > 2$, then  actually $\abs{\UKPlus / \UKctv} \geq 4$. Let $1, \ve_2, \ve_3, \ve_4 \in\UKPlus$ lie in different classes modulo $\UKctv$. Then every universal lattice contains $\qf{1,\ve_2,\ve_3,\ve_4}$ as a sublattice (using Lemma \ref{Lemma:UnitsSplit} and the indecomposability of units), and hence it is at least quaternary.
\end{proof}

\section{Fields with one nonsquare positive unit}
\label{Sec:Cases}

In this section we work with lattices. Recall that all lattices are assumed to be classical and totally positive definite. Our main goal is to gain enough knowledge about the case $\abs{\UKPlus/\UKctv}=2$, so that we can later prove Theorem \ref{Th:MainKitSqs} in case \ref{Th:MainKitSqs-units}. The results are summarized in the following theorem, which we will prove at the end of the section.

\begin{theorem}\label{Th:MainConditions}
Let $K$ be a totally real number field such that $\abs{\UKPlus/\UKctv}=2$ with the system of representatives $1, \ve$.  Assume that there exists a universal ternary lattice $L$ over $K$. Then:
\begin{enumerate}
    \item $L\simeq\qf{1,\ve} \perp L_u$ where $L_u$ is a unary lattice. \label{Th:MainConditions-Ldiag}
     \item If $L$ is free, then $\qf{1, 1, \ve, \ve}$ is universal. \label{Th:MainConditions-freeL}
    \item For each indecomposable $\lambda\in\OKPlus$, we have $\lambda=\square$ or $\ve\lambda=\square$. \label{Th:MainConditions-idecomposables}
    \item For each indecomposable $\lambda\in\OKPlus$, we have $\lambda=\square$ or $2\lambda=\square$. \label{Th:MainConditions-idecomposables2}
    \item $2\ve=\square$; in particular, $2$ is ramified in $K$. \label{Th:MainConditions-TwoEpsilon} 
    \item $\sqrt2\notin K$. \label{Th:MainConditions-NoSqrt2}
    \item $2\OKPlus \subset \sum\square$. \label{Th:MainConditions-SumOfSquares}
    \item For every $\alpha \in \OKPlus$ with $\norm{K/\Q}{\alpha}<2^{[K:\Q]}$, we have that $\norm{K/\Q}{\alpha}$ is a power of $2$. \label{Th:MainConditions-Norm}
\end{enumerate}
\end{theorem}

Note that part \ref{Th:MainConditions-SumOfSquares} of Theorem \ref{Th:MainConditions} already says that every element of $2\OKPlus$ can be written as the sum of squares;  however, we do not know how many squares are necessary. Hence, this is not enough to prove Theorem \ref{Th:MainKitSqs} in case \ref{Th:MainKitSqs-units}. For that, the most important part of Theorem \ref{Th:MainConditions} is \ref{Th:MainConditions-NoSqrt2}, which provides the crucial implication: \emph{If $\abs{\UKPlus/\UKctv}\geq2$ and there exists a universal ternary quadratic form (or, more generally, lattice) over $K$, then $\sqrt2\notin K$.} In other words, in Theorem \ref{Th:MainKitSqs}, condition \ref{Th:MainKitSqs-units} implies condition \ref{Th:MainKitSqs-sqrt2}.

Until the end of this section, we will use the following assumption:

\vspace{-3mm}
\begin{equation}\label{AssumLat}
\parbox{0.9\textwidth}{%
\begin{mdframed}
\noindent $K$ is a totally real number field with $\abs{\UKPlus/\UKctv}=2$ and with fixed $\ve\in\UKPlus\setminus\UKctv$, and there exists a universal (classical) ternary lattice $L$ over $K$.
\end{mdframed}
}
\tag{{\smaller[2]{\faTree}}}
\end{equation}

The following lemma is a reformulation of Theorem \ref{Th:MainConditions}\ref{Th:MainConditions-Ldiag}, and it follows directly from Lemma \ref{Lemma:UnitsSplit}.

\begin{lemma} \label{Lemma:Ldiag}
Suppose \eqref{AssumLat}. Then $L\simeq\qf{1,\ve}\perp L_u$ for a unary lattice $L_u$.
\end{lemma}

We proceed to part \ref{Th:MainConditions-freeL} of Theorem \ref{Th:MainConditions}.

\begin{proposition} \label{Prop:freeL}
Suppose \eqref{AssumLat} and assume that $L$ is free. Then $\qf{1,1,\ve,\ve}$ is universal.
\end{proposition}

\begin{proof}
The assumption that $L$ is free together with Lemma \ref{Lemma:Ldiag} implies that $L\simeq\qf{1,\ve,\gamma}$ for some $\gamma\in\OKPlus$. Consider the representation of $\ve\gamma$ by $L$ as $x^2+\ve y^2 + \gamma z^2$ with $x,y,z\in\OK$. We have $\ve\gamma \succeq \gamma z^2$, giving $\ve \succeq z^2$. Since $\ve$ is indecomposable and nonsquare, we get $z=0$. Thus, $\gamma = \ve(\ve^{-1}x)^2 + y^2$. Therefore, $\qf{\gamma}$ is represented by $\qf{1,\ve}$, thus the universal lattice $\qf{1,\ve,\gamma}$ is represented by $\qf{1,\ve,1,\ve}$.
\end{proof}

Note that it is not obvious how to extend the proof of Proposition \ref{Prop:freeL} to the case of non-free lattices. Instead, we move on to a description of  the indecomposable elements of $\OKPlus$ (part \ref{Th:MainConditions-idecomposables} of Theorem \ref{Th:MainConditions}). 

\begin{lemma} \label{Lemma:EpsilonIndecompSquare}
Suppose \eqref{AssumLat}. Then for each indecomposable $\lambda\in\OKPlus$, we have $\lambda=\square$ or $\ve\lambda=\square$.
\end{lemma}

\begin{proof}
Recall that $L\simeq \qf{1,\ve}\perp L_u$ by Lemma \ref{Lemma:Ldiag}. Let $\lambda\in\OKPlus$ be indecomposable; then $\ve\lambda$ is also indecomposable, and $\lambda\to L$ and $\ve\lambda\to L$. Thanks to indecomposability, $\lambda$ and $\ve\lambda$ are each represented by one of the unary lattices $\qf{1}$, $\qf{\ve}$, $L_u$. Since $\ve\lambda^2\neq\square$, $\lambda$ and $\ve\lambda$ cannot be both represented by $L_u$ due to Lemma \ref{Lemma:unary}, and so at least one of them is represented by $\qf{1}$ or $\qf{\ve}$. Each of the four possibilities gives $\lambda = \square$ or $\ve\lambda=\square$.
\end{proof}

The previous lemma can be used to prove a weaker analogy of Proposition \ref{Prop:freeL} without the assumption that the universal ternary lattice is free: If $K$ is as in \eqref{AssumLat}, then the diagonal form $\qf{1, \ldots, 1} \perp \qf{\ve, \ldots,\ve}$ of rank $\cP(\OK) + \cP(\OK) = 2\cP(\OK)$ is universal; here, $\cP(\OK)$ denotes the \emph{Pythagoras number} of $\OK$.

\smallskip

Now we attempt to prove part \ref{Th:MainConditions-TwoEpsilon} of Theorem \ref{Th:MainConditions}; however, we get two options, and only later we will prove that the case $2=\square$ is impossible (which will also prove part \ref{Th:MainConditions-NoSqrt2} of the same theorem). The following lemma is \cite[Thm.~1.2]{KTZ} generalized for not-necessarily-free lattices. 

\begin{lemma} \label{Lemma:TwoOrTwoEpsilon}
Suppose \eqref{AssumLat}. Then $2\ve=\square$ or $2=\square$. In particular, $2$ is ramified in $K$.
\end{lemma}

\begin{proof} 
We proceed by contradiction and assume that neither $2$ nor $2\ve$ is a square.

By Lemma \ref{Lemma:Ldiag}, $L\simeq\qf{1,\ve}\perp L_u$. Consider a representation of $2$ by $L$. By Lemma \ref{Lemma:2and3}\ref{Lemma:2and3-2}, the only decomposition of $2$ into elements of $\OKPlus$ is $2=1+1$. Thus, since $2\to L$, there are two possibilities: $2$ is represented by one of the unary forms $\qf{1}$, $\qf{\ve}$, and $L_u$, or $1$ is represented by two of them. In the former case, $2\to\qf{1}$ is equivalent to $2=\square$ and $2\to\qf{\ve}$ to $2\ve=\square$; thus, neither can happen and the only option is $2\to L_u$. In the latter case, of course $1\to\qf{1}$, and $1\not\to\qf{\ve}$ as $\ve$ is not a square; thus, $2$ is represented in this way if and only if $1 \to L_u$. To summarize, $2\to L$ if and only if $1\to L_u$ or $2\to L_u$.

A simple corollary of Lemma \ref{Lemma:2and3}\ref{Lemma:2and3-2} is that the only decomposition of $2\ve$ as the sum of totally positive integers is $2\ve=\ve+\ve$. Thus, by the same arguments as above, we also get that $2\ve \to L$ if and only if $\ve\to L_u$ or $2\ve\to L_u$.

By Lemma \ref{Lemma:unary}, if $\alpha_1 \to L_u$ and $\alpha_2 \to L_u$, then $\alpha_1\alpha_2 = \square$. This means that one of the following products is a square: $1\cdot\ve$, $2\cdot\ve$, $1\cdot2\ve$, $2\cdot2\ve$. This is the desired contradiction ($4\ve=\square$ would yield $\ve=\square$).
\end{proof}

To be able to exploit the condition $2\ve=\square$ from Lemma \ref{Lemma:TwoOrTwoEpsilon}, we prepare the following two results that can be stated slightly more generally.

\begin{lemma} \label{Lemma:TwoIndecompSquare}
Suppose \eqref{AssumLat} and that $a\ve=\square$ for some positive integer $a$. Then for each indecomposable $\lambda\in\OKPlus$, we have $\lambda=\square$ or $a\lambda=\square$.
\end{lemma}

\begin{proof}
Let $\lambda\in\OKPlus$ be indecomposable. By Lemma \ref{Lemma:EpsilonIndecompSquare}, we have $\lambda=\square$ or $\ve\lambda=\square$. Since $a\ve=\square$ by the assumption, the latter case is equivalent to $a\lambda=\square$ (we use that $\OK$ is integrally closed).
\end{proof}

This easily implies the following strong condition.

\begin{lemma} \label{Lemma:TwoEpsSquareSumSquares}
Suppose \eqref{AssumLat} and that $a\ve=\square$ for some positive integer $a$. Then $a\OKPlus \subset \sum\square$.
\end{lemma}

\begin{proof}
Observe that by Lemma \ref{Lemma:TwoIndecompSquare}, for every indecomposable $\lambda\in\OKPlus$, we know that $a\lambda=\square$  or $a\lambda=\lambda+\dots+\lambda=\square+\dots+\square$ is the sum of $a$ squares. Every element $\alpha\in\OKPlus$ can be written as the sum of indecomposables, and by multiplying this decomposition by $a$ we get $a\alpha$ expressed as the sum of squares.
\end{proof}

From now until the end of this section, we assume that $\sqrt2 \in K$, and prove a series of statements which culminates in a contradiction, thus proving Theorem \ref{Th:MainConditions}\ref{Th:MainConditions-NoSqrt2}.

Since $\sqrt2\in K$, a natural attempt would be to consider the representation of $2+\sqrt2$ which is indecomposable thanks to its norm. (Every totally positive element of norm $<2^{[K:\Q]}$ is indecomposable by Lemma \ref{Lemma:smallAreIndecomposable}.) Thus, $2+\sqrt2$ or $\ve(2+\sqrt2)$ is a square by Lemma \ref{Lemma:EpsilonIndecompSquare}. But as this leads to a distinction of two subcases, we do not follow this path: we can extract more information by considering the representations of $3$ and $3\ve$. Using Lemma \ref{Lemma:2and3}, we prove that to represent both $3$ and $3\ve$, one of these two numbers must be a square. Later we will get a contradiction in both cases.

\begin{lemma} \label{Lemma:ThreeEpsilonSquare}
Suppose \eqref{AssumLat} and that $\sqrt2\in K$. Then $3=\square$ or $3\ve=\square$.
\end{lemma}

\begin{proof}
Recall that $L\simeq\qf{1,\ve}\perp L_u$ for a unary lattice $L_u$ by Lemma \ref{Lemma:Ldiag}. We consider the representation of $3$ by $L$: we have $3 = x^2+\ve y^2 + \mu$ where $x,y\in\OK$ and $\mu \to L_u$. This is a decomposition of $3$ as the sum of three totally nonnegative integers. According to Lemma \ref{Lemma:2and3}, all possible decompositions are $3$, $1+2$, $1+1+1$, and possibly $\bigl(\frac{1+\sqrt5}{2}\bigr)^2 + \bigl(\frac{1-\sqrt5}{2}\bigr)^2$.

If the decomposition $3=x^2+\ve y^2 + \mu$ corresponds to $3=3$, then it contains only one nonzero summand, i.e., $3$ is represented by one of the three unary forms. If 3 is represented by $x^2$ or $\ve y^2$, then we obtain the conclusion of the lemma that $3=\square$ or $3\ve=\square$. The only remaining case is that $3$ is represented by $L_u$.

Assume now that we have one of the other decompositions of 3. Each of them contains at least two nonzero squares as summands (as 2 is a square in $K$). However, $\ve y^2\neq\square$ (because $\ve\neq\square$), and so $L_u$ must represent a square.

To summarize the previous two paragraphs, if neither $3$ nor $3\ve$ is a square, then $L_u$ represents a square or it represents $3$. Also, if we apply the same argument to the representation of $3\ve$, it similarly turns out that if neither $3$ nor $3\ve$ is a square, then $L_u$ represents an element of the form $\ve\square$ or $L_u$ represents $3\ve$.

This gives four possible combinations of elements which are simultaneously represented by $L_u$; applying Lemma \ref{Lemma:unary}, we see that one of the following products is a square: $\square\cdot\ve\square$, $\square\cdot3\ve$, $3\cdot\ve\square$, $3\cdot3\ve$. 
Since $\ve$ is not a square, we have proved our claim.
\end{proof}

We immediately strengthen the previous lemma by ruling out the latter case.

\begin{lemma} \label{Lemma:SqrtTwoImpliesSqrtThree}
Suppose \eqref{AssumLat} and $\sqrt2\in K$. Then $\sqrt3 \in K$. (That is, $3\ve = \square$ is impossible in Lemma~$\ref{Lemma:ThreeEpsilonSquare}$.)
\end{lemma}

\begin{proof}
Assume that on the contrary, $3\ve = \square$. Then by Lemma \ref{Lemma:TwoEpsSquareSumSquares} (used for $a=3$) we have $3\OKPlus \subset \sum \square$. 

If we reduce $3\OKPlus \subset \sum \square$ modulo $2\OK$ and use the identity $x^2+y^2 \equiv (x+y)^2 \pmod{2\OK}$, we obtain that 
every $\alpha \in \OKPlus$ is a square modulo ${2\OK}$.

As every element $\alpha \in \OK$ can be made totally positive by adding a sufficiently large even rational integer, we further get that in fact, every $\alpha \in \OK$ is a square modulo ${2\OK}$.

Note that $2$ ramifies since $\sqrt2 \in K$, and so let us take a prime ideal $\ip$ such that $\ip^2 \mid 2\OK$. Consider an arbitrary element $\beta \in \ip \setminus \ip^2$; then $\beta$ is not a square modulo $\ip^2$. (Indeed, if $\beta \in \ip$ and $\beta\equiv \gamma^2\pmod{\ip^2}$, then $\gamma^2 \in \ip$ and thus $\gamma\in\ip$, contradicting $\beta\notin\ip^2$.) 
Because $\ip^2 \mid 2\OK$, then $\beta$ is also not a square modulo $2\OK$, which is a contradiction.
\end{proof}

Let us recapitulate the situation: Now we know that if \eqref{AssumLat} and $\sqrt2\in K$, then $K$ must contain the field $\Q(\!\sqrt2,\sqrt3)$. Our next goal is to show that this is impossible. We will crucially exploit the fact that $\abs{\UKPlus / \UKctv}$ is \emph{exactly} $2$.

Let us start with some preparation on signatures of units. 
Observe that when $K$ is a totally real field of degree $d$, then $\U_K \simeq \{\pm1\}\times\Z^{d-1}$ by Dirichlet's unit theorem, and so $\abs{\U_K/\U_K^2} =2^d$. Thanks to the chain of subgroups $\U_K^2\subset\UKPlus\subset \U_K$, we have that $\UKPlus / \UKctv$ is a subgroup of $\U_K/\U_K^2$, and so its size is $\abs{\UKPlus / \UKctv}=2^k$ for some integer $0\leq k\leq d$.

\begin{lemma}\label{Lemma:signatures}
Let $K$ be a totally real number field of degree $d$, and let $0\leq k\leq d$ be the integer such that $\abs{\UKPlus / \UKctv}=2^k$. Then:
 \begin{enumerate}
     \item In the group of all signatures $\{\pm1\}^d$, the signatures which correspond to units form a subgroup of index $2^k$. \label{Lemma:signatures-1}
     \item Assume $k=1$. If we write $\OK \setminus \{0\}$ as a disjoint union of two sets: $M_+=\OKPlus\U_K$ and $M_-=\OK \setminus (M_+\cup\{0\})$, then the following holds: If $\beta_1,\beta_2\in M_-$, then $\beta_1\beta_2\in M_+$. \label{Lemma:signatures-2}
 \end{enumerate}    
\end{lemma}

\begin{proof}
The result is quite straightforward and folklore, but let us include its proof for completeness. First, we have to formalize part \ref{Lemma:signatures-1}. 
Let 
\[\begin{array}{cccc}
\sg:& K\setminus\{0\}& \longrightarrow& \{\pm1\}^d\\
& \alpha & \mapsto & (\sign\sigma_1(\alpha), \dots, \sign\sigma_d(\alpha))
\end{array}\]
be the group homomorphism which sends an element to its signature. 

For the restriction $\sg: \U_K \to \{\pm1\}^d$, we have $\ker \sg = \UPlus_K$. Thus:
\[
\sg(\U_K) \simeq \quot{\U_K}{\ker\sg} = \quot{\U_K}{\UPlus_K} \simeq \quot{\left( \U_K/\U_K^2 \right)}{\left( \UPlus_K/\U_K^2 \right)}.
\]
Therefore
\[\abs{\sg(\U_K)}=\abs{{\U_K}/{\UPlus_K}}=\frac{\abs{\U_K/\U_K^2}}{\abs{\UPlus_K/\U_K^2}}=\frac{2^d}{2^k}=2^{d-k},\]
and so $\sg(\U_K)$ is indeed a subgroup of $\{\pm1\}^d$ of index $2^k$.

Part \ref{Lemma:signatures-2} follows from \ref{Lemma:signatures-1}: if $k=1$, then the group $\{\pm1\}^d$ decomposes into exactly two cosets modulo $\sg(\U_K)$. We can thus let
\[
M_+ = \{\alpha \in \OK\setminus\{0\} \mid \sg(\alpha) \in \sg(\U_K)\}, \quad
M_- = \{\alpha \in \OK\setminus\{0\} \mid \sg(\alpha) \not\in \sg(\U_K)\},
\]
and the rest is clear.
\end{proof}

We use Lemma \ref{Lemma:signatures} to show that, in our situation, the norm of all small elements is a power of two. This will help us to rule out some fields.

\begin{proposition} \label{Prop:descentGeneralised}
Let $K$ be a totally real number field of degree $d$ satisfying
\begin{itemize}
    \item $\abs{\UKPlus / \UKctv} = 2$ with $\ve\in\UPlus_K\setminus\UKctv$; and
    \item for each indecomposable element $\lambda\in\OKPlus$, it holds that $\lambda=\square$ or $\ve\lambda=\square$; and
    \item $2=\square$ or $2\ve=\square$.
\end{itemize}
Then the following hold:
 \begin{enumerate}
     \item For every $\alpha \in \OKPlus$ with $\norm{K/\Q}{\alpha}<2^d$, $\norm{K/\Q}{\alpha}$ is a power of $2$. \label{Prop:descentGeneralised-norm}
     \item $\sqrt6\notin K$. \label{Prop:descentGeneralised-six}
 \end{enumerate}
\end{proposition}

\begin{proof}
First note that \ref{Prop:descentGeneralised-six} follows from \ref{Prop:descentGeneralised-norm}: if $\sqrt6\in K$, then $3+\sqrt6\in K$ is totally positive and $\norm{K/\Q}{3+\sqrt6}=3^{d/2}<2^d$.

For part \ref{Prop:descentGeneralised-norm}, we denote $A=\{2\}\cup\{\lambda\in\OKPlus \mid \lambda \text{ is indecomposable}\}$. By the assumptions on $K$,  for every $\alpha\in A$, we have either $\alpha=\square$, or $\ve\alpha=\square$ (but not both, as $\ve\neq\square$). Hence, we can define:
\[\begin{array}{cccc}
T:& A & \longrightarrow & \OK\\
&\alpha & \mapsto & \begin{cases}
                        \sqrt{\alpha} & \text{if } \alpha=\square,\\
                        \sqrt{\ve\alpha} & \text{if } \ve\alpha=\square.
                    \end{cases}
\end{array}\]
Note that $\abs{\norm{K/\Q}{T(\alpha)}}=\sqrt{\norm{K/\Q}{\alpha}}$.

We will keep the notation of Lemma \ref{Lemma:signatures}: we denote $M_+=\OKPlus\U_K$ and $M_-=\OK \setminus (M_+\cup\{0\})$. Furthermore, for each $\alpha\in M_+$, we fix some element  of $\eta_\alpha\in \UK$ that satisfies $\eta_\alpha\alpha\in\OKPlus$.

Now, as a preliminary step, we will show that there exists an element $\beta \in M_-$ with $\abs{\norm{K/\Q}{\beta}}=2^j \leq 2^{d/2}$. We start by taking $\alpha_0 = 2$. Clearly, $\alpha_0 \in A$ and we can apply $T$. Now, if $T(\alpha_0)\in M_-$, we can put $\beta = T(\alpha_0)$. Otherwise, we put $\alpha_1= \eta_{T(\alpha_0)}T(\alpha_0)$; we have $\alpha_1\in\OKPlus$ and it is indecomposable by Lemma \ref{Lemma:smallAreIndecomposable}; hence, $\alpha_1\in A$. We apply the same argument with $\alpha_0$ replaced by $\alpha_1$. Since the norms are integers and they decrease in each step, eventually $T(\alpha_n)$ must lie in $M_-$, and we set $\beta=T(\alpha_n)$.

For a contradiction, assume that there exists $\alpha\in \OKPlus$ with $\norm{K/\Q}{\alpha}<2^d$, the norm of which is not a power of $2$. Then there is an odd prime $p$ such that $v_p\bigl(\norm{K/\Q}{\alpha}\bigr)>0$. Fix one such $p$ and consider the set
\[
S = \{\delta \in \OKPlus \mid \norm{K/\Q}{\delta}<2^d \text{ and } v_p\bigl(\norm{K/\Q}{\delta}\bigr)>0\}.
\]
It is nonempty, as $\alpha\in S$. Note that all elements of $S$ are indecomposable by Lemma \ref{Lemma:smallAreIndecomposable}. Now consider some $\gamma\in S$ with minimal $v_p\bigl(\norm{K/\Q}{\gamma}\bigr)$, and denote $\tau = T(\gamma)$.

If $\tau \in M_+$, then $\norm{K/\Q}{\eta_\tau\tau}= \sqrt{\norm{K/\Q}{\gamma}}$, so $\eta_\tau \tau \in S$ and 
\[v_p\bigl(\norm{K/\Q}{\eta_\tau \tau}\bigr) = \frac12 v_p\bigl(\norm{K/\Q}{\gamma}\bigr)<v_p\bigl(\norm{K/\Q}{\gamma}\bigr),\] 
which contradicts the choice of $\gamma$.

If $\tau \in M_-$, then $\beta\tau \in M_+$ by Lemma \ref{Lemma:signatures}. Set $\gamma_0 = \eta_{(\beta\tau)}\beta\tau$. Then $\gamma_0\in\OKPlus$, and
\[
\norm{K/\Q}{\gamma_0} = \abs{\norm{K/\Q}{\beta}\norm{K/\Q}{\tau}} \leq 2^{d/2} \textstyle{\sqrt{\norm{K/\Q}{\gamma}}} < 2^{d/2}\sqrt{2^d} = 2^d.
\]
Furthermore,
\[
v_p\bigl(\norm{K/\Q}{\gamma_0}\bigr)=v_p\bigl(\pm1\cdot\norm{K/\Q}{\beta}\norm{K/\Q}{\tau}\bigr)=v_p\bigl(\norm{K/\Q}{\tau}\bigr)=\frac12 v_p\bigl(\norm{K/\Q}{\gamma}\bigr),
\]
so $\gamma_0\in S$ and $v_p\bigl(\norm{K/\Q}{\gamma_0}\bigr)<v_p\bigl(\norm{K/\Q}{\gamma}\bigr)$, which again contradicts the choice of $\gamma$.
\end{proof}

Note that the conditions of Proposition \ref{Prop:descentGeneralised} are fulfilled by every field that is as in \eqref{AssumLat}, thanks to Lemma \ref{Lemma:EpsilonIndecompSquare} and Lemma \ref{Lemma:TwoOrTwoEpsilon}. We also know that if such a field contains $\sqrt2$, then it also contains $\sqrt3$ by Lemma \ref{Lemma:SqrtTwoImpliesSqrtThree}, and hence also $\sqrt6$. As the latter is impossible by the previous proposition, we obtain the following corollary.

\begin{corollary} \label{Cor:noSqrtTwo}
Let $K$ be as in \eqref{AssumLat}. Then $\sqrt2\notin K$.
\end{corollary}

In other words, fields with at least one nonsquare unit that contain $\sqrt2$ do not admit a universal ternary lattice. However, $\sqrt2$ is not the only element with such power.

\begin{corollary}
Let $E\supset K$ be totally real number fields such that  $\abs{\UKPlus/\UKctv}=2$. Suppose that there exists $\beta\in\OKPlus$ such that $\norm{K/\Q}{\beta}<2^{[K:\Q]}$ and $\norm{K/\Q}{\beta}$ is not a power of $2$. 
Then there does not exist a universal ternary lattice over $E$. 
\end{corollary}

\begin{proof}
We first prove the claim for the field $K$. Assume that there exists a universal ternary lattice over $K$; then $K$ is as in \eqref{AssumLat}, and hence, as explained before, it satisfies the conditions of Proposition \ref{Prop:descentGeneralised}. That contradicts the existence of $\beta$.

Now let $E/K$ be a field extension. The main observation here is the well-known fact that $\abs{\UKPlus[E]/\UKctv[E]}\geq \abs{\UKPlus[K]/\UKctv[K]}$ (see, e.g., \cite[Rem.~1]{DDK}). If $\abs{\UKPlus[E]/\UKctv[E]}>2$, then the statement follows from Proposition \ref{Prop:TooManyUnits}. If $\abs{\UKPlus[E]/\UKctv[E]}=2$, then we can proceed as in the previous paragraph, as $\norm{E/\Q}{\beta}<2^{[E:\Q]}$ and $\norm{E/\Q}{\beta}$ is still not a power of $2$.
\end{proof}

In particular, we see that no totally real number field containing $\Q(\!\sqrt6)$ or $\Q(\!\sqrt{33})$ admits a universal ternary lattice,  as these quadratic fields contain a nonsquare totally positive unit as well as an element of norm $3<2^2$.

Finally, we have collected all the necessary ingredients for the main theorem of this section.

\begin{proof}[Proof of Theorem $\ref{Th:MainConditions}$]
The assumptions ensure that $K$ satisfies \eqref{AssumLat}. Hence: part \ref{Th:MainConditions-Ldiag} follows from Lemma \ref{Lemma:Ldiag}, part \ref{Th:MainConditions-freeL} was proven in Proposition \ref{Prop:freeL}, and part \ref{Th:MainConditions-idecomposables} is the content of Lemma \ref{Lemma:EpsilonIndecompSquare}.

Furthermore, we have $2=\square$ or $2\ve=\square$ by Lemma \ref{Lemma:TwoOrTwoEpsilon}; however, $2=\square$ is impossible by Corollary \ref{Cor:noSqrtTwo}. This proves parts \ref{Th:MainConditions-TwoEpsilon} and \ref{Th:MainConditions-NoSqrt2}. Moreover, we can now use Lemma \ref{Lemma:TwoIndecompSquare} to prove part \ref{Th:MainConditions-idecomposables2} and Lemma \ref{Lemma:TwoEpsSquareSumSquares} for part \ref{Th:MainConditions-SumOfSquares} (in both cases, taking $a=2$). 

Finally, note that all the conditions of Proposition \ref{Prop:descentGeneralised} are fulfilled, and so part \ref{Th:MainConditions-Norm} follows.
\end{proof}

\section{Fields without \texorpdfstring{$\sqrt 2$}{sqrt2}}\label{Sec:KitaokaSquares}

In this section, we work with free lattices, i.e., with quadratic forms. Recall that all quadratic forms are assumed to  be classical and totally positive definite. Furthermore, we impose the following:

\vspace{-3mm}
\begin{equation}\label{AssumForm}
\parbox{0.9\textwidth}{%
\begin{mdframed}
\noindent $K$ is a totally real number field such that $\sqrt2\notin K$, and there exists a universal (classical) ternary quadratic form $Q$ over $K$.
\end{mdframed}
}
\tag{{\smaller[2]{\faStar}}}
\end{equation}
\smallskip

The main goal of this section is to prove Theorem \ref{Th:MainKitSqs} under assumption \ref{Th:MainKitSqs-sqrt2} (which is the same as \eqref{AssumForm}). However, instead of proving that every element of $2\OKPlus$ is represented by the sum of four squares, we will show that it is represented by the form $\qf{1,1,2,2}$. As $\qf{2,2}\to\qf{1,1}$ (this representation follows from the identity $2x^2+2y^2=(x+y)^2+(x-y)^2$), the claim of Theorem \ref{Th:MainKitSqs} follows.

\begin{theorem} \label{Th:KitaokaSquares}
Suppose \eqref{AssumForm}. Then every element of $2\OKPlus$ is represented by $\qf{1,1,2,2}$.
\end{theorem}

\begin{proof}
The case when $Q$ is diagonalizable is handled in Theorem \ref{Th:diagonal}\ref{Th:diagonal-1}; non-diagonalizable $Q$ is covered by Proposition \ref{Prop:nondiag} combined with Lemma \ref{Lemma:nondiagShape}. 
\end{proof}

As suggested in the proof above, we distinguish between the cases when the universal ternary form is diagonalizable and when it is not.  We will slightly abuse the terminology and say that a quadratic form $Q$ is \emph{diagonal} when it is \emph{diagonalizable}, i.e., when $Q\simeq\qf{\alpha_1,\dots,\alpha_n}$ for some $\alpha_1,\dots,\alpha_n\in\OKPlus$.

\subsection{Diagonal forms} \label{Subsec:Diag}

In this subsection we show that if $K \not\ni \sqrt2$ is a totally real number field admitting a diagonal universal ternary quadratic form, then all of $2\OKPlus$ is represented by $\qf{1,1,2,2}$. Note that we make no assumptions on the (non)existence of a nonsquare totally positive unit.

\begin{lemma} \label{Lemma:diagonal}
Let $K$ be a totally real number field not containing $\sqrt2$ and $Q$ a diagonal ternary quadratic form over $K$ representing $1$ and $2$. Then $Q \simeq \qf{1,1,\alpha}$ or $\qf{1,\gamma,\alpha}$ for some $\alpha\in\OKPlus$ and $\gamma\in\OKPlus$ such that $2=\gamma t^2$ for $t\in\OK$.
\end{lemma}

\begin{proof}
A diagonal ternary form representing $1$ can be written as $\qf{1,\alpha_1,\alpha_2}$. The only decompositions of $2$ are $1+1$ and $2+0$ by Lemma \ref{Lemma:2and3}\ref{Lemma:2and3-2}. Since $2$ is not a square, $\qf{1}$ does not represent $2$, so $\qf{\alpha_1,\alpha_2}$ represents $1$ or $2$. If it represents $1$, then $Q$ is isometric to $\qf{1,1,\alpha}$ for some $\alpha\in\OKPlus$ just as claimed; if it does not represent $1$ but represents $2$, then we may assume $2 \to \qf{\alpha_1}$. This means $2=\alpha_1 t^2$ exactly as we needed.
\end{proof}

We first examine the form $\qf{1,1,\alpha}$. Here, as a side note, one might deduce quite a lot of other information -- e.g., $\alpha$ must be a nonsquare indecomposable (unless $K$ is $\Q(\!\sqrt5)$); further, every nonsquare indecomposable is of the form $\alpha t^2$; and by considering the representation of $3$, one learns that if $\sqrt3,\sqrt5\notin K$, then $2$ or $3 \to \qf{\alpha}$. The former case is covered by the forms $\qf{1,\gamma,\alpha}$ of the other type, and the latter leads to a contradiction (it would yield that $2$ is unramified, which contradicts Theorem \ref{Th:unramified}). However, for now, our aim is a clean proof that if the form is universal, then all of $2\OKPlus$ is represented by $\qf{1,1,2,2}$; this is contained in the next proposition.

\begin{proposition} \label{Prop:11alpha}
Suppose \eqref{AssumForm} and $Q\simeq \qf{1, a ,\alpha}$ for some $\alpha\in\OKPlus$ and $ a \in\{1,2\}$. Then every element of $2\OKPlus$ is represented by $\qf{1,1,2,2}$.
\end{proposition}

\begin{proof}
First of all, if 
$K=\Q$ or $\Q(\!\sqrt5)$, then  the claim holds, as $\Q$ does not admit a universal ternary form, while over $\Q(\!\sqrt5)$, the form $\qf{1,1,2}$ is universal by \cite{CKR}. Assume now that $K\neq \Q,\Q(\!\sqrt5)$. 

Consider the representation of $2\alpha$ by $Q$: we have $2\alpha = x^2+ a  y^2+\alpha z^2$ for some $x,y,z\in\OK$. Rearranged, this yields
\[
\alpha(2-z^2) = x^2+ a  y^2;
\]
in particular, since all the other terms are totally nonnegative, we get $2-z^2 \succeq 0$. This yields $z^2 \in \{0,1,2\}$ (by Lemma \ref{Lemma:2and3}), and as $2$ is not a square, $z^2\neq 2$. 

If $z^2=1$, then $\alpha = x^2+ a  y^2$. As $ a \in\{1,2\}$, it is the sum of two squares, and so $\alpha \to \qf{1,1,1}$ and $Q \to \qf{1,1,1,1,1,1}$. 
In particular, the sum of six squares is universal -- but this is impossible by \cite{Si} (as we are assuming $K\neq \Q,\Q(\!\sqrt5)$).

Thus $z^2=0$ and $2\alpha = x^2+ a  y^2$. Now, since every element of $2\OKPlus$ can be represented by the form $2Q$, it can be written as
\[
2X^2+2 a  Y^2+2\alpha Z^2 = 2X^2+2 a  Y^2+(x^2+ a  y^2)Z^2. 
\]
If $ a =1$, this yields a representation by $\qf{2,2,1,1}$. If $ a =2$, we get a representation by $\qf{2,1,1,2}$.
\end{proof}

The rest of this subsection is strictly speaking unnecessary, since Proposition \ref{Prop:nondiag} covers all forms of the shape $\qf1 \perp Q_0$ where $Q_0$ represents $2$, regardless of diagonalizability; however, in the present proof we get some extra information about the diagonal case, see Theorem \ref{Th:diagonal}\ref{Th:diagonal-2}.

\begin{proposition} \label{Prop:1gammaalpha}
Assume \eqref{AssumForm} and $Q\simeq \qf{1,\gamma,\alpha}$ for some $\alpha,\gamma\in\OKPlus$ such that $2=\gamma t^2$ for some $t\in\OK$. Then every element of $2\OKPlus$ is represented by $\qf{1,1,2,2}$.
\end{proposition}

\begin{proof}
This time we consider the representation of $\gamma\alpha$ by $Q$. First, we get
\[
\alpha(\gamma-z^2) = x^2+\gamma y^2
\]
for some $x,y,z\in\OK$, so $\gamma-z^2 \succeq 0$. Thus, $2-(tz)^2 \succeq 0$, and hence $(tz)^2$ is $0$ or $1$. 

If $(tz)^2=1$, then $t$ is a unit and $\qf{\gamma}\simeq\qf{2}$; this case has been covered in Proposition \ref{Prop:11alpha} (for $ a =2$). 

If $(tz)^2=0$, then $z=0$, and we have
\[
\gamma\alpha = x^2 + \gamma y^2;
\]
multiplication by $t^2$ yields $2\alpha = (tx)^2+2y^2$. Again, every element of $2\OKPlus$ can be represented by the form $2Q\simeq\qf{2,2\gamma,2\alpha}$; therefore, it can be written as
\[
2X^2+2\gamma Y^2+2\alpha Z^2 = 2X^2+(\gamma t Y)^2+\bigl((tx)^2+2y^2\bigr)Z^2 \to \qf{2,1,1,2}.\qedhere
\]
\end{proof}

Now we are ready to prove the main theorem of this subsection.

\begin{theorem} \label{Th:diagonal}
Suppose \eqref{AssumForm} and that $Q$ is diagonal. Then:
 \begin{enumerate}
     \item Every element of $2\OKPlus$ is represented by $\qf{1,1,2,2}$. \label{Th:diagonal-1}
     \item $Q\simeq\qf{1,1,\alpha}$ where $2\alpha$ is the sum of two squares, or $Q\simeq\qf{1,\gamma,\alpha}$ where $2=\gamma t^2$ for some $t\in\OK$ and $2\alpha = x^2+2y^2$ for $x,y\in\OK$. \label{Th:diagonal-2}
 \end{enumerate}
\end{theorem}
\begin{proof}
Lemma \ref{Lemma:diagonal} gives the only two possible shapes of the diagonal universal ternary quadratic form $Q$ (under the assumption $\sqrt2\notin K$). These are then handled in Propositions \ref{Prop:11alpha} and \ref{Prop:1gammaalpha}.

The second part is obtained directly by inspecting the proofs. We do not list $\Q$ and $\Q(\!\sqrt5)$ as exceptions: $\Q$ does not admit a universal ternary form, and for $\Q(\!\sqrt5)$, one easily checks that all the diagonalizable ternary forms, listed in \cite{CKR}, are of the desired type.
\end{proof}

\subsection{Non-diagonalizable forms} \label{Subsec:Nondiag}
Now we turn to the case when the universal ternary quadratic form $Q$ is not diagonal.

\begin{lemma} \label{Lemma:nondiagShape}
Suppose \eqref{AssumForm} and that $Q$ is not diagonalizable. Then $Q\simeq\qf{1}\perp Q_0$ with $2\to Q_0$.
\end{lemma}

\begin{proof}
Note that since $Q$ is universal, it represents $1$ and $2$; therefore, it can be written as $\qf{1} \perp Q_0$ by Lemma \ref{Lemma:UnitsSplit}. The only decompositions of $2$ are $2+0$ and $1+1$ by Lemma \ref{Lemma:2and3}\ref{Lemma:2and3-2}. Since $2\neq\square$ by the assumption, we have $2 \not\to \qf{1}$. Thus, $Q_0$ represents $1$ or $2$. If $1\to Q_0$, then $Q\simeq\qf{1,1,\alpha}$; in particular, it is diagonalizable. Hence, $2 \to Q_0$.
\end{proof}

The theory we develop in the rest of this subsection works more generally for every universal ternary quadratic form of the shape $\qf1 \perp Q_0$ where $2\to Q_0$. We need the following lemma on two-dimensional free lattices over Dedekind domains. It is a slight generalization of \cite[Lemma~5.6]{KTZ}; for completeness, we provide a proof.

\begin{lemma} \label{Lemma:pseudobasis}
Let $(\OK^2,Q_0)$ be a free lattice representing $\alpha\in\OKPlus$. Then there exist vectors $\e \in \OK^2$ with $Q_0(\e)=\alpha$ and $\f\in K^2$ and an (integral) ideal $\ia$ with $\ia^2 \supset (\alpha)$ such that $\OK^2 = \ia^{-1}\e + \ia\f$. Further, if we denote $\beta = B_{Q_0}(\e,\f)$ and $\gamma=Q_0(\f)$, then:
 \begin{enumerate}
     \item $\alpha \in \ia^2$, $\beta \in \OK$ and $\gamma \in \ia^{-2}$, \label{Lemma:pseudobasis-a}
     \item $\alpha\gamma-\beta^2 \in \OKPlus$, \label{Lemma:pseudobasis-b}
     \item $\alpha y+\beta z \in \ia$ for every $y \in \ia^{-1}$ and $z \in \ia$. \label{Lemma:pseudobasis-c}
 \end{enumerate}
\end{lemma}
\begin{proof}
First we prove the part which does not concern the quadratic structure, namely: \emph{For every nonzero $\e \in \OK^2$, there exists an $\f\in K^2$ such that $\OK^2 = \ia^{-1}\e + \ia\f$ for an ideal $\ia$; moreover, if $\e=\begin{psmallmatrix} e_1 \\ e_2 \end{psmallmatrix}$ and $\f=\bigl(\begin{smallmatrix} f_1 \\ f_2 \end{smallmatrix}\bigr)$, then $\ia=(e_1,e_2)$ and $\ia^{-1}=(f_1,f_2)$.}

For a vector $\vv = \begin{psmallmatrix} v_1 \\ v_2 \end{psmallmatrix}\in K^2$, denote $\mathcal{I}_{\vv}$ the set of all $\alpha\in K$ such that $\alpha\vv \in \OK^2$. 
Clearly, this is a fractional ideal, and in fact, $\mathcal{I}_{\vv}=(v_1,v_2)^{-1}$ almost by definition. Now, our claim is precisely \cite[Lemma~4.25]{KTZ} if we take $\ia=\mathcal{I}_{\f}$; this is an integral ideal as it is equal to $(e_1,e_2)$.

Now we consider the quadratic structure on $\OK^2$. We pick $\e$ so that $Q_0(\e)=\alpha$, and find the corresponding $\f$ and $\ia$. Proofs of all the remaining facts are based solely on the knowledge that $e_1,e_2\in\ia$ and $f_1,f_2\in\ia^{-1}$ and that $Q_0$ and $B_{Q_0}$ are polynomials with coefficients in $\OK$. (It is important that we assume all lattices to be classical.)

First, $\alpha = Q_0(\e)\in\OK e_1^2+\OK e_1e_2+\OK e_2^2$, so $\alpha \in \ia^2$. This is equivalent to $\ia^{2} \supset (\alpha)$. By an analogous argument, $\gamma = Q_0(\f)$ yields $\gamma \in \ia^{-2}$, and similarly (using bilinearity and the fact that $B_{Q_0}$ takes integral values on the standard basis vectors), $\beta=B_{Q_0}(\e,\f)\in \OK e_1f_1+\OK e_1f_2+\OK e_2f_1+\OK e_2f_2 = \ia\ia^{-1} = \OK$. This concludes the proof of \ref{Lemma:pseudobasis-a}. 

As for \ref{Lemma:pseudobasis-b}, observe $\alpha\gamma \in \ia^2\ia^{-2} = \OK$. This, together with $\beta\in\OK$, yields $\alpha\gamma-\beta^2 \in \OK$; further, the expression is totally positive by Cauchy--Schwarz inequality. Part \ref{Lemma:pseudobasis-c} follows directly from manipulation with fractional ideals.
\end{proof}

As an application of Lemma \ref{Lemma:pseudobasis}, we get a description of non-diagonalizable ternary quadratic forms representing $1$ and $2$.

\begin{lemma} \label{Lemma:shapeOfNondiag}
Suppose \eqref{AssumForm} and that $Q \simeq \qf{1}\perp Q_0$ where $2 \to Q_0$. Then there exists an ideal $\ia$ with $\ia^2 \supset (2)$ and elements $\beta\in\OK$, $\gamma\in\ia^{-2}$ with $\gamma\succ0$ such that every value represented by $Q_0$ is equal to $2y^2+2\beta yz + \gamma z^2$ where $y \in \ia^{-1}$ and $z \in \ia$, and further:
 \begin{enumerate}
     \item $2\gamma-\beta^2 \in \OKPlus$,
     \item $2y+\beta z \in \ia \subset \OK$.
 \end{enumerate}
\end{lemma}

\begin{proof}
The claim is a direct application of Lemma \ref{Lemma:pseudobasis} with the choice $\alpha=2$: every element of $\OK^2$ can be written as $y\e+z\f$ for some $y\in\ia^{-1}$ and $z\in\ia$, and we have $Q_0(y\e+z\f) = 2y^2+2\beta yz + \gamma z^2$.
\end{proof}

Of course, the most natural (and, if $2$ is unramified, the only) case is when $\ia=\OK$; then $\gamma \in \OKPlus$ and $Q_0$ is isometric to the quadratic form $2x_1^2+2\beta x_1x_2+\gamma x_2^2$. However, in general we do not know that the vector representing $2$ can be completed to give a basis. For that, it is not enough to assume $\sqrt2\notin K$ or even that $2$ is squarefree; there can still be an ideal such that its square divides $(2)$. An illustration of this can be found in \cite[Ex.~A.1]{KTZ}.

\begin{proposition} \label{Prop:nondiag}
Suppose \eqref{AssumForm} and that $Q \simeq \qf{1}\perp Q_0$ where $2 \to Q_0$. Then every element of $2\OKPlus$ is represented by $\qf{1,1,2,2}$. 
\end{proposition}

\begin{proof}
By Lemma \ref{Lemma:shapeOfNondiag}, we know what $Q_0$ looks like. Denote $\delta = 2\gamma - \beta^2 \in \OKPlus$. Consider the representation $\delta \to Q$: there exist $x\in\OK$, and (by Lemma \ref{Lemma:shapeOfNondiag}) $y \in \ia^{-1}$ and $z \in \ia$ such that $\delta = x^2+2y^2+2\beta yz + \gamma z^2$. Multiplying by $2$ and completing the square, we get
\[
2\delta = 2x^2 + (2y+\beta z)^2 + \delta z^2,
\]
so
\[
\delta(2-z^2) = 2x^2+(2y+\beta z)^2.
\]
This means that $2-z^2 \succeq 0$, so $z^2\in\{0,1,2\}$; however, $2$ is impossible as it is not a square.

If $z=\pm1$, then first observe that $\ia=\OK$, since $z\in\ia$. This also yields $y\in\OK$. We get $\delta = 2x^2+(2y\pm\beta)^2 \to \qf{2,1}$.

In the other case, $z=0$, so we get $\delta = x^2+2y^2$; however, this does not necessarily mean $\delta \to \qf{1,2}$, as $y \in \ia^{-1}$.

Now consider any $\alpha \in \OKPlus$. We need to show $2\alpha \to \qf{1,1,2,2}$. Since we assume $Q$ to be universal, we have $\alpha \to Q$. Therefore, there exist $X \in \OK$, $Y\in\ia^{-1}$ and $Z\in\ia$ such that
\[
2\alpha = 2X^2 + (2Y+\beta Z)^2 + \delta Z^2.
\]
Since $2Y+\beta Z \in \OK$ by Lemma \ref{Lemma:shapeOfNondiag}, the first two terms are represented by $\qf{2,1}$. It remains to prove that $\delta Z^2 \to \qf{1,2}$. This is clear in the first case when $\delta$ itself is represented by $\qf{1,2}$. In the latter case,
\[
\delta Z^2 = (x^2+2y^2)Z^2 = (xZ)^2 + 2(yZ)^2 \to \qf{1,2};
\]
indeed, $yZ \in \ia^{-1}\ia = \OK$.
\end{proof}

\subsection{Indecomposables and consequences}

Using analogous arguments as in Subsections~\ref{Subsec:Diag} and~\ref{Subsec:Nondiag} with $2$ replaced by an indecomposable element, we can prove a variation on Theorem \ref{Th:KitaokaSquares}. Note that we do not assume $\sqrt2 \notin K$.

\begin{proposition}\label{prop:lambda}
Let $K$ be a totally real number field that admits a universal ternary quadratic form. Let $\lambda$ be a nonsquare indecomposable element. Then every element of $\lambda\OKPlus$ can be represented by $\qf{1,1,\lambda,\lambda}$.
\end{proposition}

\begin{proof}
Let $Q$ be a universal ternary quadratic form over $K$. Then $Q$ represents $1$ and $\lambda$; as in the proof of Lemma \ref{Lemma:nondiagShape}, we get $Q\simeq\qf{1}\perp Q_0$ with $\lambda \to Q_0$. 

Using Lemma \ref{Lemma:pseudobasis} with the choice $\alpha=\lambda$ (as in Lemma \ref{Lemma:shapeOfNondiag}), we get that there exists an ideal $\ia$ with $\ia^2\supset(\lambda)$ such that every value represented by $Q_0$ is equal to $\lambda y^2+2\beta yz+\gamma z^2$,  where $y\in\ia^{-1}$, $z\in\ia$, $\beta\in\OK$, $\gamma\in\ia^{-2}$; furthermore, $\lambda\gamma-\beta^2\in\OKPlus$ and $\lambda y+\beta z\in\ia$.

Then we proceed analogously as in the proof of Proposition \ref{Prop:nondiag}: considering the representation of $\delta=\lambda\gamma-\beta^2$, we get
\[
\delta(\lambda-z^2)=\lambda x^2+(\lambda y+\beta z)^2
\]
with $x\in\OK$, and $z, \lambda y+\beta z\in\ia$. Since $\lambda$ is indecomposable and not a square, the only possibility is $z=0$, i.e., $\delta=x^2+\lambda y^2$. Considering $\alpha\in\OKPlus$, we find some $X\in\OK$, $Y\in\ia^{-1}$, and $Z\in\ia$, such that 
\[\lambda\alpha=\lambda X^2+(\lambda Y +\beta Z)^2 +\delta Z^2.\]
As $\lambda X^2+(\lambda Y +\beta Z)^2\to\qf{\lambda,1}$, and $\delta Z^2=(xZ)^2+\lambda(yZ)^2\to\qf{1,\lambda}$, the claim follows.
\end{proof}

In particular, we recover the statement of Theorem \ref{Th:MainConditions}\ref{Th:MainConditions-freeL}: if $K$ admits a universal ternary quadratic form and there exists a nonsquare unit $\ve$, then $\qf{1,1,\ve,\ve}$ is universal.

Theorem \ref{Th:KitaokaSquares} is interesting on its own, and it also provides a simple computational tool for proving that a given field fails to admit a universal ternary quadratic form -- instead of checking all possible universal ternary quadratic forms or performing the escalation procedure, it is often enough to check representation of small elements of $2\OKPlus$ (e.g., $2$ times an indecomposable) by $\qf{1,1,2,2}$. However, it also has some interesting consequences -- the main one being Theorem \ref{Th:unramified}, the proof of which we postpone until Section \ref{Sec:Proof}. See also Section \ref{Sec:Deg4} for applications of Theorem \ref{Th:KitaokaSquares} to fields of degree four. For fields of higher degree, one gets the following by applying \cite[Thm.~3.2]{KY-EvenBetter}.

\begin{remark}
Suppose \eqref{AssumForm}. Then $K=\Q(\alpha_1,\ldots,\alpha_n)$ where $\alpha_i\in\OK$ and $\house{\alpha_i}<2+\sqrt6$ for all $i$. (Here, $\house{\alpha}=\max_j \abs{\sigma_j(\alpha)}$ is the \emph{house} of $\alpha$.) 
\end{remark}

For the other corollary of Theorem \ref{Th:KitaokaSquares}, we first need the following observation. 

\begin{proposition}\label{Prop:unramified}
Let $K$ be a totally real number field where all of $2\OKPlus$ is represented by $\qf{1,1,2,2}$. Then $2$ is ramified in $K$ unless $K=\Q$ or $\Q(\!\sqrt5)$.
\end{proposition}

\begin{proof}
We show that if $2$ is unramified and all of $2\OKPlus$ is represented by $\qf{1,1,2,2}$, then $\qf{1,1,1,1}$ is universal. Invoking Siegel's theorem on the non-universality of the sum of squares \cite{Si}, this will be sufficient.

Let $\alpha\in\OKPlus$. Then
\[
2\alpha = x^2+y^2+2z^2+2w^2
\]
for some $x,y,z,w\in\OK$. Considering this equality modulo $2$, we get $0 \equiv (x+y)^2$; and since $2$ is unramified, this yields $x+y \equiv 0$. Thus we can write $\alpha$ as the sum of four integral squares as follows:
\[
\alpha = \Bigl(\frac{x+y}{2}\Bigr)^2 + \Bigl(\frac{x-y}{2}\Bigr)^2 + z^2 + w^2. \qedhere
\]
\end{proof}

Note that Proposition \ref{Prop:unramified} also implies that even in the cubic field $\Q(\zeta_{7}+\zeta_{7}^{-1})$ with discriminant $49$, one of the few known fields where $2\OKPlus \subset \sum\square$, not every element of $2\OKPlus$ is represented by $\qf{1,1,2,2}$.

\section{Fields of degree four} \label{Sec:Deg4}

In this section, we apply our results to fields of degree four.

\begin{proposition}\label{Prop:zeta20}
The field $\Q(\zeta_{20}+\zeta_{20}^{-1})$ does not admit a universal ternary lattice.
\end{proposition}
\begin{proof}
This field is generated by $\vartheta = \zeta_{20}+\zeta_{20}^{-1}=\sqrt{\frac{5+\sqrt5}{2}}$ with minimal polynomial $x^4-5x^2+5$. It contains a nonsquare unit $\ve=\vartheta+2\in\UPlus_K$ and a totally positive element $p_5=\vartheta(\vartheta-1)$ of norm $5<2^4$, which is thus indecomposable. But neither $p_5$ nor $\ve p_5$ is a square (since their norm is not a square); therefore, we get from Theorem \ref{Th:MainConditions}\ref{Th:MainConditions-idecomposables} that over $\Q(\zeta_{20}+\zeta_{20}^{-1})$, there does not exist a universal ternary lattice.
\end{proof}

\begin{remark}
\setlist[enumerate]{wide=0pt, itemsep=3pt}
\begin{enumerate} 
    \item Note that the first two simple criteria which might have proved the nonexistence of a universal ternary lattice did not provide an answer: 1) $2\ve=\square$, so this necessary condition is satisfied; 2) numerical evidence suggests that every element of $2\OKPlus$ is indeed represented by $\qf{1,1,2,2}$. On the other hand, we also could have applied Theorem \ref{Th:MainConditions}\ref{Th:MainConditions-Norm}, as $p_5$ is a totally positive element whose norm is $<2^4$ but not a power of two. 
    \item By an announced result of Krásenský--Scharlau, the field $\Q(\zeta_{20}+\zeta_{20}^{-1})$ does admit a non-classical universal ternary quadratic form. This universal form can be written as $x^2 + y^2 + \ve z^2 + xy + \pi yz$, where $\ve=\vartheta+2$ as above (with $\vartheta=\zeta_{20}+\zeta_{20}^{-1}$) and $\pi=\vartheta^3+\vartheta^2-3\vartheta-2$ is the dyadic prime which satisfies $2\ve=\pi^2$.
\end{enumerate}
\end{remark}

\begin{corollary} \label{Cor:MainDegFour}
Let $K$ be a totally real number field of degree four. Assume that at least one of the following holds:
\begin{enumerate}
    \item $\sqrt2\notin K$ or \label{Cor:MainDegFour-notin}
    \item $\abs{\UKPlus/\UKctv}\geq2$. \label{Cor:MainDegFour-in} 
\end{enumerate}
Then there is no universal ternary quadratic form over $K$. 
\end{corollary}

\begin{proof}
Assume that $K$ admits a universal ternary form. If $\sqrt2\notin K$, then every element of $2\OKPlus$ is represented by $\qf{1,1,2,2}$  by Theorem \ref{Th:KitaokaSquares}; therefore, $2\OKPlus\subset\sum\square$. In the case when $\abs{\UKPlus/\UKctv}\geq2$, we can restrict to $\abs{\UKPlus/\UKctv}=2$ by Proposition \ref{Prop:TooManyUnits}. Hence, we can apply Theorem \ref{Th:MainConditions}\ref{Th:MainConditions-SumOfSquares}, to get $2\OKPlus\subset\sum\square$ again. 

Invoking \cite[Thm.~1.1]{KY-EvenBetter}, the only options are $K=\Q(\!\sqrt2,\sqrt5)$ and $K=\Q(\zeta_{20}+\zeta_{20}^{-1})$. However, the former field is ruled out by \cite[Thm.~1.1]{KTZ} and the latter by Proposition \ref{Prop:zeta20}.
\end{proof}


\section{Proofs of main theorems}\label{Sec:Proof}

Combining Theorem \ref{Th:KitaokaSquares} and Proposition \ref{Prop:unramified}, we can prove Theorem \ref{Th:unramified}. 

\begin{proof}[Proof of Theorem $\ref{Th:unramified}$]
If $2$ is unramified, then $\sqrt2\notin K$, so Theorem \ref{Th:KitaokaSquares} applies and $\qf{1,1,2,2}$ represents all of $2\OKPlus$. Then, by Proposition \ref{Prop:unramified}, we have $K=\Q$ or $K=\Q(\!\sqrt5)$. It is well known that $\Q(\!\sqrt5)$ admits a universal ternary quadratic form, while $\Q$ does not.
\end{proof}

Let us also summarize how Theorem \ref{Th:MainKitSqs} follows from what we have established.

\begin{proof}[Proof of Theorem $\ref{Th:MainKitSqs}$]
Let $K$ be a totally real number field, and assume that there exists a universal ternary classical quadratic form over $K$.

Suppose that \ref{Th:MainKitSqs-sqrt2} holds, i.e., $\sqrt2\notin K$. By Theorem \ref{Th:KitaokaSquares}, every element of $2\OKPlus$ is represented by $\qf{1,1,2,2}$. Since $2x^2+2y^2=(x+y)^2+(x-y)^2$, we have $\qf{2,2}\to\qf{1,1}$, and hence also $\qf{1,1,2,2}\to\qf{1,1,1,1}$. In particular, every element of $2\OKPlus$ can be written as the sum of four squares.

Assume \ref{Th:MainKitSqs-units}, i.e., that $\abs{\UKPlus/\UKctv}\geq 2$. Proposition \ref{Prop:TooManyUnits} leaves only the case $\abs{\UKPlus/\UKctv}= 2$. Then we can apply Theorem \ref{Th:MainConditions}\ref{Th:MainConditions-NoSqrt2} to get that $\sqrt2\notin K$. Therefore, condition \ref{Th:MainKitSqs-sqrt2} holds, and the claim follows from the previous part of the proof.

Finally, by Corollary \ref{Cor:MainDegFour}, no such field $K$ can exist in degree four. 
\end{proof}

\printbibliography

\end{document}